\documentclass[sumlimits,intlimits,
reqno,12pt]{amsart}
\usepackage{amsmath, amsfonts, amssymb,amsthm}
\usepackage{fancyhdr}\usepackage[pdftex,colorlinks]{hyperref}
\usepackage{enumerate}
\theoremstyle{plain}
\newtheorem{theorem}{Theorem}[section]
\newtheorem{lemma}[theorem]{Lemma}
\newtheorem{corollary}[theorem]{Corollary}
\newtheorem{proposition}[theorem]{Proposition}

\theoremstyle{definition}
\newtheorem{definition}[theorem]{Definition}
\newtheorem{remark}[theorem]{Remark}

\theoremstyle{remark}

\numberwithin{equation}{section}

\def\lr{\mathop{\longrightarrow}}\def\limn{\mathop{\underset{t}\varprojlim}}\def\inlim{\mathop{\underset{t}\varprojlim}}
\def\inlimk{\mathop{\underset{k}\varprojlim}}
\def\limt{\mathop{\underset{t}\varinjlim}}\def\dlim{\mathop{\underset{t}\varinjlim}}
\def\dlims{\mathop{\underset{s}\varinjlim}}

\def\lc{linearly compact}

\def\hlt{Hausdorff linearly topologized}

 \def\Ext{\mathop{\mathrm{Ext}}}
\def\Tor{\mathop{\mathrm{Tor}}}\def\Supp{\mathop{\mathrm{Supp}}}\def\Cosupp{\mathop{\mathrm{Cosupp}}}
 
\def\Ann{\mathop{\mathrm{Ann}}}

\def\Im{\mathop{\mathrm{Im}}}
\def\Ndim{\mathop{\mathrm{Ndim}}}
\def\Ass{\mathop{\mathrm{Ass}}}

\def\Cosupp{\mathop{\mathrm{Cosupp}}}

\def\Hom{\mathop{\mathrm{Hom}}}
\def\dim{\mathop{\mathrm{dim}}}
\def\p{\mathop{\frak{p}}}
\def\m{\mathop{\frak{m}}}

\begin{document}
\title{Formal local homology}
\author{TRAN TUAN NAM}
\address{Tran Tuan Nam\\Ho Chi Minh University
of Pedagogy\\ 280 An Duong Vuong, District 5, Ho Chi Minh City,
Vietnam  } \email{namtuantran@gmail.com}
\date{}
\maketitle

\markboth{\sc TRAN TUAN NAM}{Formal local homology}

\footnote{This research is funded by Vietnam National Foundation for
Science and Technology Development (NAFOSTED)}

\noindent{\bf Abstract.} We introduce a concept of  formal local
homology modules which is in some sense dual to P. Schenzel's
concept of
 formal local cohomology modules. The dual theorem and the
non-vanishing
 theorem of formal local homology  modules will be shown.   We also
give some  conditions for formal local homology modules being
finitely generated or artinian.
\medskip

%\noindent {\it 2000 Mathematics subject classification}: Local cohomology 13D45, Homological functors on modules 16E30, Topological rings and modules 13J99.

\noindent {\it Key words}: Formal local cohomology, formal local
homology.
\bigskip
\section{Introduction} \label{S:introduc}
\medskip
Throughout this paper, $(R,\m)$ will be a   local noetherian
(commutative) ring with the $\m-$adic topology. Let $I$ be an ideal
of $(R,\m)$ and $M$ an $R-$module. In \cite{schonf}, P. Schenzel
introduced the concept of {\it formal cohomology} and the $i-th$
$I-$formal cohomology module of $M$ with respect to $\m$ can be
defined by
$$\mathfrak{F}_I^{i}(M) = \limn H_{\m}^i(M/I^tM).$$
In the paper, we introduce the concept of {\it formal local
homology} which is in some sense dual to P. Schenzel's concept of
formal local cohomology. The $i-$th $I-${\it formal local homology
module} $\mathfrak{F}^I_{i, J}(M)$ of an $R-$module $M$ with respect
to $J$ is defined by
$$\mathfrak{F}^I_{i, J}(M) = \limt H^J_i(0:_M I^t).$$In the case of  $J=\m$ we set $\mathfrak{F}^I_{i,
\m}(M)=\mathfrak{F}^I_{i}(M)$ and  speak simply about the $i-$th
$I-$formal local homology module.

We also study some basic properties of formal local homology modules
$\mathfrak{F}^I_{i}(M)$ when $M$ is a \lc\ $R-$module, in particular
when $M$ is an artinian $R-$module. The organization of the paper is
as follows.

In Section 2, we  define  the formal local homology modules
$\mathfrak{F}^I_{i, J}(M)$ of an $R-$module $M$ with respect to $J.$
It is shown that $H^0_I(\mathfrak{F}^I_{j, J}(M))\cong
\mathfrak{F}^I_{j, J}(M)$ and $H^i_I(\mathfrak{F}^I_{j, J}(M))=0$
for all $i\not=0$ (Theorem \ref{T:ddddflb0}). The dual theorem
(Theorem \ref{T:dldnddddfm}) establishes the isomorphisms
$$\mathfrak{F}^I_{i}(M^*)\cong \mathfrak{F}_{I}^{i}(M)^*,\
\mathfrak{F}_I^{i}(M^*)\cong \mathfrak{F}^{I}_{i}(M)^*$$ provided
  $M$ is a \lc\ module over the  complete ring $(R,\m).$  In
Theorem \ref{T:dkhdddfmartin} the short exact sequence of artinian
modules $0\longrightarrow M' \longrightarrow M \longrightarrow M''
\longrightarrow 0$ gives rise to a long exact sequence of $I-$formal
local homology modules
$$... \longrightarrow \mathfrak{F}^I_{i}(M') \longrightarrow
\mathfrak{F}^I_{i}(M) \longrightarrow \mathfrak{F}^I_{i}(M'')
\longrightarrow \mathfrak{F}^I_{i-1}(M')\longrightarrow ....$$ This
section is closed by the non-vanishing theorem (Theorem
\ref{L:dlkttdddpfm}) which says that if $M$ is a non zero
semi-discrete linearly compact $R-$module such that $0\leq \Ndim
(0:_MI) \not=1,$ then
$$\Ndim (0:_MI) = \max \{ i
\mid \mathfrak{F}^I_{i}(M) \not=0 \}.$$ On the other hand
 if $M$ is a semi-discrete \lc\
$R-$module such that $\Ndim (0:_{\Gamma_{\m}(M)}I)\not= 0,$ then
 $$\Ndim (0:_{\Gamma_{\m}(M)}I) = \max \big\{ i\mid
\mathfrak{F}^I_{i}(M)\not=0\big\}.$$

The last section is devoted to study the finiteness  of formal local
homology modules. Theorems \ref{L:dktthhsrdcddfm} and
\ref{L:lhsdktthhsrdcddfm} give us the equivalent conditions for the
formal local homology modules $\mathfrak{F}^I_{i}(M)$ being finitely
generated. In Theorem \ref{L:dfmhhshomt},  if  $I$ is a principal
ideal of $(R,\m)$ and $M$  an artinian $R-$module, then
$\mathfrak{F}^I_{i}(M)/I\mathfrak{F}^I_{i}(M)$ is a noetherian
$\widehat{R}-$module for all $i.$ Theorem \ref{L:khacdfmhhshomt}
shows that if  $M$ is  an artinian $R-$module and $s$ a non-negative
integer such that $\mathfrak{F}^I_{i}(M)$ is  a noetherian
$\widehat{R}-$module for all $i<s,$ then
$\mathfrak{F}^I_{s}(M)/I\mathfrak{F}^I_{s}(M)$ is also a noetherian
$\widehat{R}-$module. There is a question: When are the formal local
homology modules $\mathfrak{F}^I_{i}(M)$ artinian? Theorem
\ref{T:ddfmnttchd} answers that if $M$ is an artinian $R-$module
with $\Ndim M = d,$ then $\mathfrak{F}^I_{d-1}(M)$ is an artinian
$R-$module.
 Finally, Theorem \ref{L:lh0lhsdktthhsrdcddfm}
provides that if  $M$ is an artinian $R-$module and $s$ a
non-negative integer, then the following statements are equivalent:
(i) $\mathfrak{F}^I_{i}(M)$ is artinian for all $i>s,$ (ii)
$\mathfrak{F}^I_{i}(M) = 0$ for all $i>s$ and
$\Ass(\mathfrak{F}^I_{i}(M))\subseteq \{\m\}$  for all $i>s.$
\bigskip

\section{Formal local homology modules}\medskip
\def\s{\mathcal{S}}\def\t{\mathcal{T}}

We first recall the concept of {\it linearly compact} defined by I.
G. Macdonald \cite{macdua}. A Hausdorff linearly topologized
$R-$module $M$ is said to be {\it linearly compact}  if
$\mathcal{F}$ is a family of closed cosets (i.e., cosets of closed
submodules) in $M$ which has the finite intersection property, then
the cosets in $\mathcal{F}$ have a non-empty intersection. A
Hausdorff linearly topologized $R-$module M is called {\it
semi-discrete} if every submodule of $M$ is closed. Thus a discrete
$R-$module is semi-discrete. It is clear that artinian $R-$modules
are linearly compact with the discrete topology. So the class of
semi-discrete linearly compact modules contains all artinian
modules. Moreover, if $(R,\m)$ is a complete ring, then the finitely
generated $R-$modules are also linearly compact and semi-discrete.

Let $I$ be an ideal of $(R,\m)$ and $M$  an $R-$module. It is
well-known that the $i-th$ local cohomology module  $H^i_I (M)$ of
$M$ with respect to  $I$ can be defined by
$$H^i_I (M) = \underset{t}{\underrightarrow{\lim}} {\Ext}_R^i (R/I^t ; M).$$
When $i=0,$ $H^0_I(M) \cong  \underset{t>0}\cup (0:_MI^t)
=\Gamma_I(M).$

 In \cite{schonf}, P. Schenzel introduced the concept of {\it
formal cohomology} and the $i-th$ $I-$formal cohomology module of
$M$ with respect to $\m$ can be defined by
$$\mathfrak{F}_I^{i}(M) = \limt H_{\m}^i(M/I^tM).$$
Note that the  $i-th$ local homology module $H^I_i (M)$ of an
$R-$module $M$ with respect to $I$ can be is defined  by
$$H^I_i (M) = \underset{t}{\underleftarrow{\lim}}{\Tor}^R_i (R/I^t , M)\ \text{(\cite{cuothe})}.$$
When $i=0,$ $H^I_0(M) \cong \inlim M//I^tM = \Lambda_I(M)$ the
$I-$adic completion of $M.$
 This suggests the following definition.

\begin{definition} \label{D:dndddpfm} Let $I,J$ be ideals of $R.$ The $i-$th $I-${\it formal local homology module} $\mathfrak{F}^I_{i, J}(M)$ of an
$R-$module $M$ with respect to $J$ is defined by
$$\mathfrak{F}^I_{i, J}(M) = \limt H^J_i(0:_M I^t).$$
\end{definition}

In the case of  $J=\m$ we set $\mathfrak{F}^I_{i,
\m}(M)=\mathfrak{F}^I_{i}(M)$ and  speak simply about the $i-$th
$I-$formal local homology module.

\begin{remark}\label{R:lydnfm}
(i). It should be mentioned from \cite[3.1 (i)]{cuoalo}  that
$H^J_i(0:_M I^t)$ has a natural structure as a module over the ring
$\Lambda_J(R),$ then $\mathfrak{F}^I_{i, J}(M)$ also has a natural
structure as a module over the ring $\Lambda_J(R).$ In particular,
$\mathfrak{F}^I_{i}(M)$ has a natural structure as a module over the
ring $\widehat{R}.$

(ii). If $M$ is finitely generated, then $H^J_i(0:_M I^t)=0$ for all
$i>0$ by \cite[3.2 (ii)]{cuothe}, then $\mathfrak{F}^I_{i, J}(M)=0$
for all $i>0.$
\end{remark}

In the following theorem we compute the local cohomology modules of
an $I-$formal local homology module $\mathfrak{F}^I_{i, J}(M).$

\begin{theorem}\label{T:ddddflb0}
Let $M$ be an  $R-$module. Then
$$H^i_I(\mathfrak{F}^I_{j, J}(M))\cong \begin{cases} 0 & i\not= 0\\
 \mathfrak{F}^{I}_{j, J}(M) & i=0
\end{cases}$$
for any integer  $j.$
\end{theorem}
\begin{proof}
We have $$H^i_I(\mathfrak{F}^I_{j, J}(M)) = H^i_I(\dlim
H^{J}_i(0:_MI^t))\cong \dlim H^i_I(H^{J}_i(0:_MI^t)).$$ Assume that
the ideal $I$ is generated by $r$ element $x_1, x_2, ..., x_r.$ Set
$\underline{x}(s) = (x_1^s, x_2^s, ..., x_r^s)$ and
$H^i(\underline{x}(s), N)$ is the $i$th Koszul cohomology module of
an $R-$module $N$ with respect to $\underline{x}(s).$ we have
$$H^i_I(\mathfrak{F}^I_{j, J}(M)) \cong \dlim \dlims
H^i(\underline{x}(s), H^{J}_j(0:_MI^t)).$$Note that
$\underline{x}(s)H^{J}_j(0:_MI^t)=0$ for all $s\geq t.$ Then
$$\dlims
H^i(\underline{x}(s), H^{J}_j(0:_MI^t))\cong \begin{cases} 0 & i\not= 0\\
 H^{J}_j(0:_MI^t) & i=0.\end{cases}$$ By passing to direct limits $\dlim$ we have
the conclusion as required.\end{proof}

\begin{corollary}\label{C:hqddddflb0} Let $M$ be an $R-$module and $i$ an integer  such
that $0:_{\mathfrak{F}^I_{i, J}(M)}I=0.$ Then $\mathfrak{F}^I_{i,
J}(M)=0.$
\end{corollary}
\begin{proof} It follows from \ref{T:ddddflb0} that
$$\mathfrak{F}^I_{i, J}(M) = \Gamma_I(\mathfrak{F}^I_{i, J}(M)) =
\underset{t>0}\bigcup(0:_{\mathfrak{F}^I_{i, J}(M)}I^t).$$ As
$0:_{\mathfrak{F}^I_{i, J}(M)}I=0,$ we conclude that
$\mathfrak{F}^I_{i,J}(M)=0.$
\end{proof}

If $M$ is a \lc\ $R-$module, then $M$ has   a natural structure of
linearly compact module over $\widehat{R}$ by \cite[7.1]{cuoalo}. We
have the following lemma.

\begin{lemma}\label{R:lydnfm}
Let $M$ be a \lc\ $R-$module. Then $$\mathfrak{F}^I_{i, J}(M) \cong
\mathfrak{F}^{I\widehat{R}}_{i, J\widehat{R}}(M)$$for all $i\geq 0.$
\end{lemma}
\begin{proof}
  The natural homomorphism $R\longrightarrow \widehat{R}$
gives by \cite[3.7]{cuothe} isomorphisms $$H^{I}_i(0:_M I^t) \cong
H^{I\hat{R}}_i(0:_M I^t\hat{R})$$ for all $i\geq 0.$ By passing to
 direct limits, we have the isomorphisms $$\mathfrak{F}^I_{i}(M)
\cong \mathfrak{F}^{I\hat{R}}_{i}(M)$$as required.
\end{proof}

It should be noted that the artinian $R-$modules are \lc\ and
discrete. Therefore we have an immediate consequence.

\begin{corollary}\label{C:hqlydnfm}
If $M$ is a artinian $R-$module, then $$\mathfrak{F}^I_{i}(M) \cong
\mathfrak{F}^{I\widehat{R}}_{i}(M)$$for all $i\geq 0.$
\end{corollary}

\begin{lemma}\label{L:cpitddfmb0} Let $I, J$ be ideals of $R$ and $M$  a \lc\ $R-$module. If
$M$ is $J-$separated (it means that $\underset{t>0}\cap J^tM=0$),
then
$$\mathfrak{F}^I_{i,J}(M)\cong \begin{cases} 0 & i\not= 0\\
 \Gamma_{I}(M) & i=0.
\end{cases}$$
\end{lemma}
\begin{proof} As $M$ is $J-$separated, $0:_MI^t$ is also
$J-$separated for all $t>0.$ It follows from \cite[3.8]{cuoalo} that
$$H^J_{i}(0:_MI^t)\cong \begin{cases} 0 & i\not= 0\\
 0:_MI^t & i=0.
\end{cases}$$ By passing to direct limits we have the conclusion.
\end{proof}

It should be noted by \cite[3.3 (i)]{cuothe} and \cite[3.3]{cuoalo}
that the local homology modules $H^J_i(M)$ are \lc\ and
$J-$separated for all $i.$ Then we have the immediate consequence.

\begin{corollary}\label{L:hqcpitddfmb0} Let $M$ be a \lc\
$R-$module. Then
$$\mathfrak{F}^I_{i,J}(H^J_j(M))\cong \begin{cases} 0 & i\not= 0\\
 \Gamma_{I}(H^J_j(M)) & i=0
\end{cases}$$for all $j.$
\end{corollary}

In the special case when $M$ is an artinian $R-$module, we have the
following consequence.

\begin{corollary}\label{C:dbatcpitddfmb0} Let $I, J$ be ideals of $R$ and $M$  an artinian $R-$module. If
$M$ is $J-$separated (it means that $\underset{t>0}\cap J^tM=0$),
then
$$\mathfrak{F}^I_{i,J}(M)\cong \begin{cases} 0 & i\not= 0\\
 M & i=0.
\end{cases}$$
\end{corollary}
\begin{proof} Note that artinian modules are \lc , then
$$\mathfrak{F}^I_{i,J}(M)\cong \begin{cases} 0 & i\not= 0\\
 \Gamma_{I}(M) & i=0
\end{cases}$$ by \ref{L:cpitddfmb0}. Moreover, as $M$ is an artinian module over the local ring $(R,\m),$
\cite[1.4]{shaame} provides that $\Gamma_I(M) = M$ and we have the
conclusion.
\end{proof}

\begin{lemma}\label{L:dfdcdcpitddfmb0} Let $I, J$ be ideals of $R$ and $M$  an artinian $R-$module. If
$M$ is $I-$separated (it means that $\underset{t>0}\cap I^tM=0$),
then
$$\mathfrak{F}^I_{i,J}(M)\cong H^J_i(M)$$for all $i.$
\end{lemma}
\begin{proof} As $M$ is a $I-$separated artinian $R-$module, there is a positive integer $n$ such that $I^nM=0.$ Then  $0:_MI^n=M.$
Therefore
$$\mathfrak{F}^I_{i, J}(M) = \limt H^J_i(0:_M I^t)\cong H^J_i(M)$$for all $i.$
\end{proof}

In the case $J=\m,$ it follows from \cite[4.6]{cuothe} that
$H^{\m}_i(M)$ is a noetherian $\widehat{R}-$module. From
\ref{L:dfdcdcpitddfmb0} we have an immediate consequence.

\begin{lemma}\label{L:dfdcdcpitddfmb0cd} Let  $M$ be  an artinian $R-$module. If
$M$ is $I-$separated (it means that $\underset{t>0}\cap I^tM=0$),
then
$$\mathfrak{F}^I_{i}(M)\cong H^{\m}_i(M)$$and then $\mathfrak{F}^I_{i}(M)$ a noetherian $\widehat{R}-$module for all $i\not= 0.$
\end{lemma}

In order to state the dual theorem we recall the concepts of Matlis
dual and Macdonald dual. Let $M$ be an $R-$module and $E(R/\m)$ the
injective envelope of $R/\m.$ The module $D(M)=\Hom(M,E(R/\m))$ is
called the Matlis dual of $M.$ If $M$ is a Hausdorff linearly
topologized $R-$module, then the {\it Macdonald dual}  of $M$ is
defined  by $M^*=Hom(M,E(R/\m))$ the set of continuous homomorphisms
of $R-$modules (\cite[\S 9]{macdua}). The topology on $M^*$ is
defined as in \cite[8.1]{macdua}. Moreover, if $M$ is semi-discrete,
then the topology of $M^*$ coincides with that induced on it as a
submodule of $E(R/\m)^M,$ where $E(R/\m)^M=\underset{x\in M}\prod
(E(R/\m))^x,$ $(E(R/\m))^x=E(R/\m)$ for all $x\in M$
(\cite[8.6]{macdua}).  A \hlt\ $R-$module $M$ is called {\it
linearly discrete} if every $\m-$primary quotient of $M$ is
discrete. It is clear that $M^*\subseteq D(M)$ and the equality
holds if and only if $M$ is semi-discrete in the following lemma.

\begin{lemma} {\rm (\cite[5.8]{macdua})}\label{L:msbmdcpttnrr}
Let $M$ be a Hausdorff linearly topologized $R-$module. Then $M$ is
semi-discrete if and only if $M^*=D(M).$
\end{lemma}

\begin{lemma}\label{L:dnmdncplrrnrrlcp} {\rm( \cite[9.3, 9.12, 9.13]{macdua})} Let
$(R, \m)$ be a complete local noetherian ring.
\begin{itemize}
\item[(i)] If $M$ is linearly compact, then $M^*$ is  linearly
discrete (hence semi-discrete). If $M$ is semi-discrete, then $M^*$
is linearly compact;
\item[(ii)] If $M$ is linearly compact or linearly discrete, then
we have a topological isomorphism $\omega
:M\overset{\simeq}\longrightarrow M^{**}$.
\end{itemize}
\end{lemma}

\begin{lemma}\label{L:mdnrrdbsaoathhs} Let
$(R, \m)$ be a complete local noetherian ring.
\begin{itemize}
\item[(i)] If $M$ is finitely generated, then
$M^*$ is artinian;
\item[(ii)] If $M$ is artinian, then $M^*$ is finitely
generated.
\end{itemize}
\end{lemma}
\begin{proof}
(i). AS  $M$ is a finitely generated module over the complete local
noetherian ring $(R, \m),$ It follows from \cite[7.3]{macdua} that
$M$ is \lc\ and semi-discrete. Then $M^*=D(M)$ by
\ref{L:msbmdcpttnrr}. Now, the conclusion follows from
\cite[3.4.11]{strhom}.

 (i). It should be noted that an artinian $R-$module is a \lc\
$R-$module with the discrete topology. Then $M^*=D(M)$ by
\ref{L:msbmdcpttnrr}. Finally, the conclusion follows from
\cite[3.4.12]{strhom}.
\end{proof}

We have the following dual theorem.

\begin{theorem}\label{T:dldnddddfm} Let $(R,\m)$ be a complete ring and $M$ a \lc\
$R-$module. Then
$$\mathfrak{F}^I_{i}(M^*)\cong  \mathfrak{F}_{I}^{i}(M)^*,$$
$$\mathfrak{F}_I^{i}(M^*)\cong  \mathfrak{F}^{I}_{i}(M)^*.$$
for for all $i.$
\end{theorem}
\begin{proof} It should be noted by \cite[6.7]{cuoalo} that
$$M^*/I^tM^*\cong (0:_M I^t)^*,$$
$$(M/I^tM)^*\cong 0:_{M^*} I^t$$for all $t>0.$ The we have
\begin{align*}
\mathfrak{F}^I_{i}(M^*) & = \limt H^{\m}_i(0:_{M^*} I^t)\\
&\cong \limt H^{\m}_i((M/I^tM)^*)\\
&\cong \limt (H_{\m}^i(M/I^tM))^*\ \text{(cf. \cite[6.4(ii)]{cuoalo})}\\
&\cong (\limn H_{\m}^i(M/I^tM))^* = \mathfrak{F}_{I}^{i}(M)^*\
\text{(cf. \cite[9.14]{macdua})}.
\end{align*}
\begin{align*}
\mathfrak{F}_I^{i}(M^*) & = \limn H_{\m}^i(M^*/I^tM^*)\\
&\cong \limn H_{\m}^i((0:_M I^t)^*)\\
&\cong \limn (H^{\m}_i(0:_M I^t))^*\ \text{(cf. \cite[6.4(ii)]{cuoalo})}\\
&\cong (\limt H^{\m}_i(0:_M I^t))^* = \mathfrak{F}^{I}_{i}(M)^*\
\text{(cf. \cite[2.6]{macdua})}.
\end{align*} The proof is complete.
\end{proof}

\begin{corollary}\label{C:hqdldnddddfm} Let $(R,\m)$ be a complete ring and $M$ a \lc\
$R-$module. Then
$$\mathfrak{F}^I_{i}(M)\cong  \mathfrak{F}_{I}^{i}(M^*)^*,$$
$$\mathfrak{F}_I^{i}(M)\cong  \mathfrak{F}^{I}_{i}(M^*)^*.$$
for for all $i.$
\end{corollary}
\begin{proof} Combining \ref{L:dnmdncplrrnrrlcp} (ii) with
\ref{T:dldnddddfm} yields
\begin{align*}
\mathfrak{F}^I_{i}(M) & \cong  \mathfrak{F}^I_{i}(M^{**})
 \cong  \mathfrak{F}_I^{i}(M^{*})^*,\\
\mathfrak{F}_I^{i}(M)& \cong \mathfrak{F}_I^{i}(M^{**}) \cong
\mathfrak{F}^{I}_{i}(M^*)^*
\end{align*}as required.
\end{proof}

In order to state Theorem \ref{T:dkhdddfmartin} about the long exact
sequence of formal local homology modules we need the following
lemma.

\begin{lemma}\label{L:dcltcsdcltddfm}
\begin{itemize}
\item[(i)] The continuous homomorphism of linearly
topology $R-$modules $f: M \longrightarrow N$ induces continuous
homomorphisms of formal local cohomology modules

$\varphi_i: \mathfrak{F}_I^{i}(M) \longrightarrow
\mathfrak{F}_I^{i}(N)$ for all
$i.$\\
\item[(ii)] If $M$ is a semi-discrete \lc\ $R-$module, then the
formal local cohomology modules $\mathfrak{F}_I^{i}(M)$ are linearly
compact for all $i.$
\end{itemize}
\end{lemma}
\begin{proof} (i).
The continuous homomorphism  $f: M \longrightarrow N$ induces
continuous homomorphisms $M/I^tM \longrightarrow N/I^tN$ for all
$t>0.$ By an  argument analogous to that used for the proof of
\cite[2.5]{cuoalo} we get continuous homomorphisms
${\Ext}^i_R(R/{\m}^s; M/I^tM) \longrightarrow {\Ext}^i_R(R/{\m}^s;
N/I^tN)$ for all $s,\ t>0$ and $i.$ By passing into direct limits
$\dlims$ we have
 continuous homomorphisms of local cohomology modules
$H^{\m}_i(M/I^tM) \longrightarrow H^{\m}_i(N/I^tN).$ Now, by passing
into inverse limits $\inlim$ we get continuous homomorphisms of
formal local cohomology modules $\varphi_i: \mathfrak{F}_I^{i}(M)
\longrightarrow \mathfrak{F}_I^{i}(N)$ for all $i.$

(ii). By \cite[7.9]{cuoalo}  the local cohomology modules
$H^i_{\m}(M/I^tM)$ are artinian for all $t$ and $i.$ Note that
$\mathfrak{F}_I^{i}(M) = \inlim H^i_{\m}(M/I^tM).$ Therefore, the
conclusion follows from \cite[3.7, 3.10]{macdua}.
\end{proof}

\begin{theorem}\label{T:dkhdddfmartin} Let $0\longrightarrow M'
\longrightarrow M \longrightarrow M'' \longrightarrow 0$ be a short
exact sequence of artinian modules. Then there is a long exact
sequence of $I-$formal local homology modules
$$... \longrightarrow \mathfrak{F}^I_{i}(M') \longrightarrow
\mathfrak{F}^I_{i}(M) \longrightarrow \mathfrak{F}^I_{i}(M'')
\longrightarrow \mathfrak{F}^I_{i-1}(M')\longrightarrow ....$$
\end{theorem}
\begin{proof} It
should be noted by \cite[1.11]{shaame} that an artinian module over
a local noetherian ring $(R,\m)$ has a natural structure of artinian
module over $\widehat{R}$ and  a subset of $M$ is an $R-$submodule
if and only if it is an $\widehat{R}-$submodule. Thus, from
\ref{C:dbatcpitddfmb0} we may assume that $(R,\m)$ is a complete
ring. We now consider the short exact sequence of artinian
$R-$modules $$0\longrightarrow M' \overset{f}\longrightarrow M
\overset{g}\longrightarrow M'' \longrightarrow 0.$$Note that the
artinian $R-$modules are \lc\ and discrete, then the homomorphisms
$f, g$ are continuous. Combining \cite[6.5]{cuoalo} with
\ref{L:mdnrrdbsaoathhs}
 we have  the following short exact sequence of
finitely generated $R-$modules
$$0\longrightarrow M''^* \overset{f^*}\longrightarrow M^* \overset{g^*}\longrightarrow M'^*
\longrightarrow  0$$ where   the induced homomorphisms $f^*,\ g^*$
are continuous.  It gives rise by \cite[3.11]{schonf} a long exact
sequence of $I-$formal local cohomology modules
$$... \longrightarrow \mathfrak{F}_I^{i-1}(M'^*)\longrightarrow   \mathfrak{F}_I^{i}(M''^*) \overset{g_i}\longrightarrow
\mathfrak{F}_I^{i}(M^*) \overset{f_i}\longrightarrow
\mathfrak{F}_I^{i}(M'^*) \longrightarrow ....$$ It should be
mentioned from \ref{L:dcltcsdcltddfm} (ii) that the $I-$formal local
cohomology modules $\mathfrak{F}_I^{i}(M^*),
\mathfrak{F}_I^{i}(M'^*), \mathfrak{F}_I^{i}(M''^*)$
 are linearly
compact and the homomorphisms of the long exact sequence are inverse
limits of the homomorphisms of artinian modules. Note that the
homomorphisms of artinian modules are continuous, because the
topologies on the artinian modules  are discrete. Moreover, the
inverse limits of the continuous homomorphisms are also continuous.
Then the homomorphisms of the long exact sequence are continuous.
 Therefore the long exact
sequence induces an exact sequence by
 \cite[6.5]{cuoalo}
$$...    (\mathfrak{F}_I^{i}(M'^*))^* \longrightarrow
(\mathfrak{F}_I^{i}(M^*))^* \longrightarrow
(\mathfrak{F}_I^{i}(M''^*))^*
 \longrightarrow (\mathfrak{F}_I^{i-1}(M'^*))^*  ....$$
The conclusion now follows from \ref{C:hqdldnddddfm}.
\end{proof}

We now recall the  concept of {\it Noetherian dimension} of an
$R-$module $M$ denoted by $\Ndim M.$ Note that the notion of
Noetherian dimension was introduced first by R. N. Roberts
\cite{robkru} by the name Krull dimension. Later, D. Kirby
\cite{kirdim} changed this terminology of Roberts and refereed to
{\it Noetherian dimension} to avoid confusion with well-know Krull
dimension of finitely generated modules.  Let $M$ be an $R-$module.
When $M=0$ we put $\Ndim M = -1.$ Then by induction, for any ordinal
$\alpha,$ we put $\Ndim M = \alpha$ when (i) $\Ndim M < \alpha$ is
false, and (ii) for every ascending chain $M_0 \subseteq M_1
\subseteq \ldots$ of submodules of $M,$ there exists a positive
integer $m_0 $  such that $\Ndim(M_{m+1} /M_m )< \alpha$ for all $m
\geq m_0$. Thus $M$ is non-zero and finitely generated if and only
if $\Ndim M = 0.$ If $0 \longrightarrow M" \longrightarrow
M\longrightarrow M'\longrightarrow 0$ is a short exact sequence of
$R-$modules, then $\Ndim M= \max\{ \Ndim M", \Ndim M'\}.$

\begin{proposition}\label{P:mdcpitddfmb0} Let $I, J$ be ideals of $R$ and $M$  an artinian $R-$module. If
$\Ndim (0:_MI)=0,$ then
$$\mathfrak{F}^I_{i,J}(M)\cong \begin{cases} 0 & i\not= 0\\
 M & i=0.
\end{cases}$$
\end{proposition}
\begin{proof} Since $\Ndim(0:_MI)=0,$  $0:_MI$ has finite length (cf. \cite[p. 269]{robkru})  and then  $0:_MI$ is $J-$separated.
It follows from \cite[3.8]{cuoalo} that
$$H^J_{i}(0:_MI^t)\cong \begin{cases} 0 & i\not= 0\\
 0:_MI^t & i=0
\end{cases}$$for all $t>0.$ By passing to direct limits we have
$$\mathfrak{F}^I_{i,J}(M)\cong \begin{cases} 0 & i\not= 0\\
 \Gamma_{I}(M) & i=0.
\end{cases}$$
Now the conclusion follows from \cite[1.4]{shaame},  as $M$ is an
artinian module over the local ring $(R,\m),$
\end{proof}

Remember that the {\it (Krull) dimension}  $\dim_R M$  of a non-zero
$R-$module $M$ is the supremum of lengths of chains of primes in the
support of $M$ if this supremum exists, and $\infty$ otherwise. For
convenience, we set $\dim M= -1$ if $M=0$.  Note that if $M$ is
non-zero and artinian, then $\dim M=0.$ If $M$ is finitely
generated, then $\dim M = \max\{\dim R/\p\mid \p\in \Ass M\}$.

In \cite[4.10]{cuoalo},  if $M$ is a non zero semi-discrete linearly
compact $R-$module, then

$$\Ndim \Gamma_{\m}(M) = \max \big\{ i\mid H^{\m}_i(M)\not=0\big\}\
\text{if}\  \Gamma_{\m}(M)\not= 0;$$
 $$\Ndim M= \max \big\{ i\mid H^{\m}_i(M)\not=0\big\}\
\text{if}\   \Ndim M\not= 1.$$

We have the non-vanishing theorem of formal local homology modules.

\begin{theorem}\label{L:dlkttdddpfm} Let $M$ be a non zero semi-discrete linearly
compact $R-$module. Then
\begin{itemize}
\item[(i)]
  $\Ndim (0:_MI) = \max \{ i
\mid \mathfrak{F}^I_{i}(M) \not=0 \}$ if $0\leq \Ndim (0:_MI)
\not=1;$
\item[(ii)] $\Ndim (0:_{\Gamma_{\m}(M)}I) = \max \big\{ i\mid
\mathfrak{F}^I_{i}(M)\not=0\big\}$\\ if $\Ndim
(0:_{\Gamma_{\m}(M)}I)\not= 0.$
\end{itemize}
\end{theorem}
\begin{proof}
(i). We begin by proving that
$$\Ndim(0:_MI^t) = \Ndim(0:_MI)$$ for all $t>0.$
As $M$ is a semi-discrete \lc\ $R-$module,  it should be noted by
\cite[7.1,7.2]{cuoalo} that $M$ has  a natural structure of
semi-discrete linearly compact module over the ring $\widehat{R}$
and $\Ndim_RM = \Ndim_{\hat{R}}M.$ Thus, we may assume that $(R,\m)$
is a complete ring.
 At first, we prove in the special case when $M$ is
artinian. Then $D(M)$ is a finitely generated $R-$module by Matlis
dual. We have $$\dim D(M)/I^tD(M) = \dim D(M)/ID(M).$$ Combining
\ref{L:msbmdcpttnrr} with \cite[7.4]{cuoalo} yields
\begin{align*}
\Ndim(0:_MI^t) & = \dim D(0:_MI^t)\\
 &= \dim D(M)/I^tD(M)\\
& = \dim D(M)/ID(M)\\
 & = \dim D(0:_MI) = \Ndim(0:_MI).
\end{align*}
We now assume that $M$ is a semi-discrete \lc\ $R-$module. By
\cite[Theorem]{zoslin} there is  a short exact sequence $$0
\longrightarrow N \longrightarrow M \longrightarrow A
\longrightarrow 0,$$ where $N$ is finitely generated and $A$ is
artinian. It induces an exact sequence
$$0 \longrightarrow 0:_NI^t \overset{f}\longrightarrow 0:_MI^t \overset{g}\longrightarrow
0:_AI^t \overset{\delta}\longrightarrow {\Ext}^1_R(R/I^t; N).$$ Then
we have two short exact sequence
$$0 \longrightarrow 0:_NI^t \overset{f}\longrightarrow 0:_MI^t \longrightarrow \Im g \longrightarrow 0,$$
$$0 \longrightarrow \Im g \longrightarrow
0:_AI^t \longrightarrow \Im \delta \longrightarrow 0.$$ Since
$0:_NI^t$ and $\Im \delta$ are finitely generated $R-$modules, we
get $\Ndim(0:_NI^t) = \Ndim \Im\delta = 0.$ It follows that
\begin{align*}
\Ndim(0:_MI^t) &= \Ndim\Im g \\
&= \Ndim(0:_AI^t) \\
&= \Ndim(0:_AI) = \Ndim(0:_MI).
\end{align*}
Now, it follows from  \cite[4.10 (ii)]{cuoalo} that
 $$d=\Ndim (0:_MI)=\Ndim (0:_MI^t)= \max \big\{ i\mid
H^{\m}_i(0:_MI^t)\not=0\big\}$$for all $t>0.$ Then
$H^{\m}_d(0:_MI^t)\not=0$ and $H^{\m}_i(0:_MI^t) = 0$ for all $i>d$
and  $t>0.$ The short exact sequences for all $t>0$
$$0\longrightarrow 0:_MI^t \longrightarrow 0:_MI^{t+1} \longrightarrow 0:_MI^{t+1}/0:_MI^t \longrightarrow
0$$induce exact sequences $$... H^{\m}_{d+1}(0:_MI^{t+1}/0:_MI^t)
 \longrightarrow H^{\m}_{d}(0:_MI^t) \longrightarrow H^{\m}_{d}(0:_MI^{t+1}) ... $$by \cite[3.7]{cuoalo}. As $\Ndim(0:_MI^{t+1}/0:_MI^t)\leq d,$
 \cite[4.8]{cuoalo} shows that $H^{\m}_{d+1}(0:_MI^{t+1}/0:_MI^t)=0.$
 Then the homomorphisms $$H^{\m}_{d}(0:_MI^t) \longrightarrow
 H^{\m}_{d}(0:_MI^{t+1})$$are injective for all $t>0.$ Therefore $\limt
 H^{\m}_{d}(0:_MI^t)\not=0$ and $\limt
 H^{\m}_{i}(0:_MI^t)=0$ for all $i>d.$ Hence (i) is proved.

 (ii). It should is clear  that $0:_{\Gamma_{\m}(M)}I =
 \Gamma_{\m}(0:_{M}I).$   Set $d=\Ndim
(0:_{\Gamma_{\m}(M)}I)=\Ndim (0:_{\Gamma_{\m}(M)}I^t).$ From
 \cite[4.10(i)]{cuoalo} we have
$$d=\Ndim (0:_{\Gamma_{\m}(M)}I^t) = \Ndim \Gamma_{\m}(0:_{M}I^t) = \max \big\{ i\mid
H^{\m}_i(0:_MI^t)\not=0\big\}.$$ The rest of the proof is
  analogous to that in the proof of (i).
\end{proof}

\begin{corollary}\label{C:hqdlkttdddpfm}  Let $M$ be an artinian $R-$module such that $0:_{M}I\not=0.$ Then
  $$\Ndim (0:_MI) = \max \{ i
\mid \mathfrak{F}^I_{i}(M) \not=0 \}.$$
\end{corollary}
\begin{proof}
It should be noted that an artinian $R-$module is semi-discrete
linearly compact. Moreover, $\Gamma_{\m}(M) = M.$ Therefore the
conclusion follows from \ref{L:dlkttdddpfm} (ii).
\end{proof}
\bigskip

\section{The finiteness  of formal local homology
modules}\label{S:msdkhhdpfm}
\medskip

We begin by recalling the concept of {\it co-support} of an module.
The co-support $\Cosupp_R(M)$ of an $R-$module $M$ is the set of
primes $\p$ such that there exists a cocyclic homomorphic image $L$
of $M$ with $\Ann(L) \subseteq \p$ (\cite[2.1]{yascoa}). Note that a
module is cocyclic if it is a submodule of $E(R/\m)$ for some
maximal ideal $\m\subset R.$

\noindent If $0\lr N \lr M \lr K \lr 0$ is an exact sequence of
$R-$modules, then $\Cosupp_R(M) = \Cosupp_R(N) \cup \Cosupp_R(K)$
(\cite[2.7]{yascoa}).

\begin{lemma}\label{L:cosddfmgvicvm} Let   $M$ be  an artinian
$R-$module. Then
$${\Cosupp}_R(\mathfrak{F}^I_{i}(M))\bigcap V(I)\subseteq V(\m).$$
\end{lemma}
\begin{proof} By \ref{R:lydnfm} we may assume that $(R,\m)$ is a complete ring It should be noted that $M^*$ is a finitely generated $R-$module. Combining \ref{C:hqdldnddddfm} with \cite[2.9]{yascoa}
yields
\begin{align*}
{\Cosupp}_R(\mathfrak{F}^I_{i}(M)) & =
{\Cosupp}_R(\mathfrak{F}_I^{i}(M^*)^*)\\
&\subseteq {\Cosupp}_RD((\mathfrak{F}_I^{i}(M^*))\\
& = {\Supp}_R(\mathfrak{F}_I^{i}(M^*).
\end{align*}
 By the proof of \cite[4.3]{schonf} we have
\begin{align*}
 {\Cosupp}_R(\mathfrak{F}^I_{i}(M))\bigcap V(I) &\subseteq {\Supp}_R(\mathfrak{F}_I^{i}(M^*)\bigcap
 V(I)\\
 &\subseteq V(\m).
\end{align*} The proof is complete.
\end{proof}

We have the following equivalent properties for formal local
homology modules $\mathfrak{F}^I_{i}(M)$ being finitely generated
for all $i<s.$

\begin{theorem}\label{L:dktthhsrdcddfm}
Let $M$ be an artinian $R-$module and $s$ a positive integer. Then
the following statements are equivalent:
\begin{itemize}
\item[(i)] $\mathfrak{F}^I_{i}(M)$ is finitely generated for all
$i<s;$
\item[(ii)] $I\subseteq \sqrt{0: {\mathfrak{F}^I_{i}(M)}}$ for all
$i<s.$
\end{itemize}
\end{theorem}
\begin{proof}
$(i)\Rightarrow (ii).$ For $i<s,$ as $\mathfrak{F}^I_{i}(M)$ is
finitely generated, the increasing chain of submodules of
$\mathfrak{F}^I_{i}(M)$
$$0:_{\mathfrak{F}^I_{i}(M)}I\subseteq 0:_{\mathfrak{F}^I_{i}(M)}I^2\subseteq ... \subseteq 0:_{\mathfrak{F}^I_{i}(M)}I^t\subseteq ...$$
is stationary. Thus, there is a positive integer $r$ such that
$0:_{\mathfrak{F}^I_{i}(M)}I^t=0:_{\mathfrak{F}^I_{i}(M)}I^r$ for
all $t\geq r.$
 It follows from \ref{T:ddddflb0} that
$$\mathfrak{F}^I_{i}(M) = \underset{t>0}\bigcup
(0:_{\mathfrak{F}^I_{i}(M)}I^t) =
0:_{\mathfrak{F}^I_{i}(M)}I^r.$$Therefore $I^r\mathfrak{F}^I_{i}(M)
= 0$ and then  $I\subseteq \sqrt{0: {\mathfrak{F}^I_{i}(M)}}.$

$(ii)\Rightarrow (i).$ We use induction on $s.$ When $s=1,$ we have
$$\mathfrak{F}^I_{0}(M) = \limt \Lambda_{\m}(0:_MI^t) =
\limt\inlimk(0:_MI^t)/{\m}^k(0:_MI^t).$$ As $0:_MI^t$ is artinian,
there is $k_t$ such that ${\m}^k(0:_MI^t) = {\m}^{k_t}(0:_MI^t)$ for
all $k\geq k_t.$ Then
$$\inlimk(0:_MI^t)/{\m}^k(0:_MI^t) = (0:_MI^t)/{\m}^{k_t}(0:_MI^t)\
\text{and}$$
\begin{align*}
\mathfrak{F}^I_{0}(M) &= \limt (0:_MI^t)/{\m}^{k_t}(0:_MI^t)\\
&\cong \limt(0:_MI^t)/\limt{\m}^{k_t}(0:_MI^t) =
M/\limt{\m}^{k_t}(0:_MI^t).\end{align*} Thus,
$\mathfrak{F}^I_{0}(M)$ is artinian and then
${\Cosupp}_R(\mathfrak{F}^I_{0}(M)) = V(0: \mathfrak{F}^I_{0}(M))$
by \cite[2.3]{yascoa}. Moreover, from the hypothesis we have
$${\Cosupp}_R(\mathfrak{F}^I_{0}(M)) \subseteq V(I).$$
 It follows from \ref{L:cosddfmgvicvm} that
$${\Cosupp}_R(\mathfrak{F}^I_{0}(M))=V(0: \mathfrak{F}^I_{0}(M))\subseteq
V(\m)$$ and then $\mathfrak{F}^I_{0}(M)$ has finite length.

Let $s>1.$ As $M$ is artinian, there is a positive integer $m$ such
that $I^tM = I^mM$ for all $t \geq  m.$ Set $K = I^mM,$ then the
short exact sequence of artinian $R-$modules $$0 \longrightarrow K
\longrightarrow M \longrightarrow  M/K \longrightarrow  0$$ gives
rise to a long exact sequence of formal local homology modules by
\ref{T:dkhdddfmartin}
$$... \longrightarrow \mathfrak{F}^I_{i+1}(M/K)\longrightarrow \mathfrak{F}^I_{i}(K)\longrightarrow \mathfrak{F}^I_{i}(M)
\longrightarrow \mathfrak{F}^I_{i}(M/K)\longrightarrow ....$$ It is
clear that $M/K$ is $I-$separated, then $\mathfrak{F}^I_{i}(M/K) =0$
for all $i>0$ and $\mathfrak{F}^I_{0}(M/K) \cong M/K$ by
\ref{C:dbatcpitddfmb0}. It follows that $\mathfrak{F}^I_{i}(K)\cong
\mathfrak{F}^I_{i}(M)$ for all $i>0$ and there is an exact sequence
$$0 \longrightarrow \mathfrak{F}^I_{0}(K)\longrightarrow \mathfrak{F}^I_{0}(M)
\longrightarrow \mathfrak{F}^I_{0}(M/K) \longrightarrow 0.$$Thus,
the proof will be complete if we show that $\mathfrak{F}^I_{i}(K)$
is finitely generated for all $i < s.$ By the hypothesis, we have
$I\subseteq \sqrt{0: {\mathfrak{F}^I_{i}(K)}}$ for all $i<s.$ Since
$IK = K,$ there is an element $x \in I$ such that $xK = K$ by
\cite[2.8]{macsec}. Then there is a positive integer $r$ such that
$x^r\mathfrak{F}^I_{i}(K)  = 0$ for all $i < s.$ Now the short exact
sequence $0 \longrightarrow 0:_K x^r \longrightarrow K
\overset{x^r}\longrightarrow K \longrightarrow 0$ induces a short
exact sequence of formal local homology modules
$$0 \longrightarrow \mathfrak{F}^I_{i}(K)\longrightarrow \mathfrak{F}^I_{i-1}(0:_K x^r)\longrightarrow \mathfrak{F}^I_{i-1}(K)\longrightarrow 0$$
for all $i < s.$ It follows $I\subseteq
\sqrt{0:\mathfrak{F}^I_{i-1}(0:_K x^r)}$ for all $i < s.$ By the
inductive hypothesis $\mathfrak{F}^I_{i-1}(0:_K x^r)$ is finitely
generated for all $i<s.$ Therefore $\mathfrak{F}^I_{i}(K)$ is
finitely generated for all $i<s$ and the proof is complete.
\end{proof}

The following theorem gives us the  equivalent properties for formal
local homology modules $\mathfrak{F}^I_{i}(M)$ being finitely
generated for all $i>s.$

\begin{theorem}\label{L:lhsdktthhsrdcddfm}
Let $M$ be an artinian $R-$module and $s$ a non-negative integer.
Then the following statements are equivalent:
\begin{itemize}
\item[(i)] $\mathfrak{F}^I_{i}(M)$ is finitely generated for all
$i>s;$
\item[(ii)] $I\subseteq \sqrt{0: {\mathfrak{F}^I_{i}(M)}}$ for all
$i>s.$
\end{itemize}
\end{theorem}
\begin{proof}
$(i)\Rightarrow (ii).$ The argument is similar to that used in the
proof of \ref{L:dktthhsrdcddfm}.

$(ii)\Rightarrow (i).$ We now proceed by induction on $d=\Ndim M.$
When $d=0,$ we have $\Ndim(0:_MI) = 0$ and then
$\mathfrak{F}^I_{i}(M)=0$ for all $i>0.$

Let $d>0.$  As $M$ is artinian, there is a positive integer $m$ such
that $I^tM = I^mM$ for all $t \geq  m.$ Set $K = I^mM,$ analysis
similar to that in the proof of \ref{L:dktthhsrdcddfm}  shows that
$$\mathfrak{F}^I_{i}(K)\cong \mathfrak{F}^I_{i}(M)$$ for all $i>0.$
Thus, The proof is completed by showing that $\mathfrak{F}^I_{i}(K)$
is finitely generated for all $i > s.$ From the hypothesis, we have
$I\subseteq \sqrt{0: {\mathfrak{F}^I_{i}(K)}}$ for all $i>s.$ Since
$IK = K,$ there is an element $x \in I$ such that $xK = K$ by
\cite[2.8]{macsec}. Then there is a positive integer $r$ such that
$x^r\mathfrak{F}^I_{i}(K) = 0$ for all $i > s.$ Set $y=x^r,$  the
short exact sequence $$0 \longrightarrow 0:_K y \longrightarrow K
\overset{y}\longrightarrow K \longrightarrow 0$$ induces a short
exact sequence of formal local homology modules
$$0 \longrightarrow \mathfrak{F}^I_{i}(K)\longrightarrow \mathfrak{F}^I_{i-1}(0:_K y)\longrightarrow \mathfrak{F}^I_{i-1}(K)\longrightarrow 0$$
for all $i > s.$ It follows $I\subseteq
\sqrt{0:\mathfrak{F}^I_{i-1}(0:_K y)}$ for all $i > s.$ It should be
noted by \cite[3.7]{cuoalo} that $\Ndim(0:_K y)\leq d-1.$ From the
inductive hypothesis $\mathfrak{F}^I_{i-1}(0:_K y)$ is finitely
generated for all $i>s.$ Therefore $\mathfrak{F}^I_{i}(K)$ is
finitely generated for all $i>s$ and the proof is complete.
\end{proof}

There is an  question: What is the module
$\mathfrak{F}^I_{i}(M)/I\mathfrak{F}^I_{i}(M)$?  When $I$ is a
principal ideal we have  an answer in the following theorem.

\begin{theorem}\label{L:dfmhhshomt}
Let $I$ be a principal ideal of $(R,\m)$ and $M$  an artinian
$R-$module. Then $\mathfrak{F}^I_{i}(M)/I\mathfrak{F}^I_{i}(M)$ is a
noetherian $\widehat{R}-$module for all $i.$
\end{theorem}
\begin{proof} Assume that $I$ is generated by the element $x.$
 As $M$ is artinian, there is a positive integer $m$ such
that $x^tM = x^mM$ for all $t \geq  m.$ Set $K = x^mM,$ then the
short exact sequence of artinian $R-$modules $$0 \longrightarrow K
\overset{f}\longrightarrow M \overset{g}\longrightarrow  M/K
\longrightarrow 0$$ gives rise to a long exact sequence of formal
local homology modules
$$... \longrightarrow \mathfrak{F}^I_{i+1}(M/K)\overset{\delta_{i+1}}\longrightarrow \mathfrak{F}^I_{i}(K)\overset{f_i}\longrightarrow
\mathfrak{F}^I_{i}(M) \overset{g_i}\longrightarrow
\mathfrak{F}^I_{i}(M/K)\longrightarrow ....$$Then we have short
exact sequences
$$0\longrightarrow \Im f_i\longrightarrow  \mathfrak{F}^I_{i}(M) \longrightarrow \Im g_i \longrightarrow 0,$$
$$0\longrightarrow \Im {\delta_{i+1}}\longrightarrow  \mathfrak{F}^I_{i}(K) \longrightarrow \Im {f_i} \longrightarrow 0.$$
These short exact sequences induce the following exact sequences
$$ \Im f_i/I\Im f_i\longrightarrow  \mathfrak{F}^I_{i}(M)/I\mathfrak{F}^I_{i}(M) \longrightarrow \Im g_i/I\Im g_i\longrightarrow 0,$$
$$\Im {\delta_{i+1}}/I\Im {\delta_{i+1}}\longrightarrow  \mathfrak{F}^I_{i}(K)/I\mathfrak{F}^I_{i}(K) \longrightarrow \Im f_i/I\Im f_i
\longrightarrow 0.$$ It should be mentioned that $M/K$ is
$I-$separated. By \ref{L:dfdcdcpitddfmb0cd},
$\mathfrak{F}^I_{i}(M/K)$ is a noetherian $\widehat{R}-$module and
then $\Im g_i/I\Im g_i$ is a noetherian $\widehat{R}-$module for all
$i.$ Thus, the proof is complete by showing that
$\mathfrak{F}^I_{i}(K)/I\mathfrak{F}^I_{i}(K)$ is a noetherian
$\widehat{R}-$module for all $i.$ As $xK = K,$ there is a short
exact sequence
$$0\longrightarrow 0:_Kx \longrightarrow K \overset{x}\longrightarrow K \longrightarrow
0.$$It gives rise to a long exact sequence $$... \longrightarrow
 \mathfrak{F}^I_{i}(0:_Kx)\longrightarrow
\mathfrak{F}^I_{i}(K)\overset{x}\longrightarrow
\mathfrak{F}^I_{i}(K)\longrightarrow
\mathfrak{F}^I_{i-1}(0:_Kx)\longrightarrow....$$Note that $0:_Kx$ is
$I-$separated, then $\mathfrak{F}^I_{i}(0:_Kx)$ is a noetherian
$\widehat{R}-$module for all $i$ by \ref{L:dfdcdcpitddfmb0cd}. It
follows from the long exact sequence that
$\mathfrak{F}^I_{i}(K)/x\mathfrak{F}^I_{i}(K)$ is a noetherian
$\widehat{R}-$module for all $i.$ Therefore
$\mathfrak{F}^I_{i}(K)/I\mathfrak{F}^I_{i}(K)$ is a noetherian
$\widehat{R}-$module for all $i$ and the proof is complete.
\end{proof}

The following theorem shows  other conditions  for the
$\widehat{R}-$module $\mathfrak{F}^I_{i}(M)/I\mathfrak{F}^I_{i}(M)$
being noetherian.

\begin{theorem}\label{L:khacdfmhhshomt}
Let  $M$ be  an artinian $R-$module and $s$ a non-negative integer.
If $\mathfrak{F}^I_{i}(M)$ is  a noetherian $\widehat{R}-$module for
all $i<s,$ then

\noindent$\mathfrak{F}^I_{s}(M)/I\mathfrak{F}^I_{s}(M)$ is a
noetherian $\widehat{R}-$module.
\end{theorem}
\begin{proof} We use induction on $s.$ When $s=0,$ it follows from
the proof of \ref{L:dktthhsrdcddfm} that $\mathfrak{F}^I_{0}(M)$ is
an artinian $R-$module. By \cite[2.3]{yascoa},
\begin{align*}
\Cosupp(\mathfrak{F}^I_{0}(M)/I\mathfrak{F}^I_{0}(M)) &= V(0:
\mathfrak{F}^I_{0}(M)/I\mathfrak{F}^I_{0}(M))\\
&= V(I+ (0: \mathfrak{F}^I_{0}(M)) \\
&= V(I)\bigcap V(0: \mathfrak{F}^I_{0}(M))\\
& = V(I)\bigcap \Cosupp(\mathfrak{F}^I_{0}(M)).
\end{align*}
It follows from \ref{L:cosddfmgvicvm} that
$\Cosupp(\mathfrak{F}^I_{0}(M)/I\mathfrak{F}^I_{0}(M))\subseteq
V(\m)$ and then $\mathfrak{F}^I_{0}(M)/I\mathfrak{F}^I_{0}(M)$ has
finite length.

Let $s>0.$  As $M$ is artinian, there is a positive integer $m$ such
that $I^tM = I^mM$ for all $t \geq  m.$ Set $K = I^mM,$ analysis
similar to that in the proof of \ref{L:dktthhsrdcddfm}  shows that
$$\mathfrak{F}^I_{i}(K)\cong \mathfrak{F}^I_{i}(M)$$ for all $i>0.$
Thus, The proof is completed by  showing that
$\mathfrak{F}^I_{s}(K)/I\mathfrak{F}^I_{s}(K)$ is a noetherian
$\widehat{R}-$module.  Since $IK = K,$ there is an element $x \in I$
such that $xK = K$ by \cite[2.8]{macsec}. Now the short exact
sequence
$$0\longrightarrow 0:_Kx \longrightarrow K \overset{x}\longrightarrow K \longrightarrow
0$$ gives rise to a long exact sequence $$...
\mathfrak{F}^I_{i}(K)\overset{x}\longrightarrow
\mathfrak{F}^I_{i}(K)\longrightarrow
\mathfrak{F}^I_{i-1}(0:_Kx)\longrightarrow
\mathfrak{F}^I_{i-1}(K)\overset{x}\longrightarrow
\mathfrak{F}^I_{i-1}(K)....$$ By the hypothesis,
$\mathfrak{F}^I_{i}(K)$ is  a noetherian $\widehat{R}-$module for
all $i<s.$ Then $\mathfrak{F}^I_{i}(0:_Kx)$ is  a noetherian
$\widehat{R}-$module for all $i<s-1.$ It follows from the inductive
hypothesis that
$\mathfrak{F}^I_{s-1}(0:_Kx)/I\mathfrak{F}^I_{s-1}(0:_Kx)$ is  a
noetherian $\widehat{R}-$module. From the long exact sequence we
have the following short exact sequence
$$0 \longrightarrow \mathfrak{F}^I_{s}(K)/x\mathfrak{F}^I_{s}(K) \longrightarrow \mathfrak{F}^I_{s-1}(0:_Kx)
\longrightarrow 0:_{\mathfrak{F}^I_{s-1}(K)}x \longrightarrow 0.$$
It induces an exact sequence {\small $${\Tor}_1^R(R/I;
0:_{\mathfrak{F}^I_{s-1}(K)}x) \longrightarrow
\mathfrak{F}^I_{s}(K)/I\mathfrak{F}^I_{s}(K)  \longrightarrow
\mathfrak{F}^I_{s-1}(0:_Kx)/I\mathfrak{F}^I_{s-1}(0:_Kx).$$} Since
$0:_{\mathfrak{F}^I_{s-1}(K)}x$ is  a noetherian
$\widehat{R}-$module, ${\Tor}_1^R(R/I;
0:_{\mathfrak{F}^I_{s-1}(K)}x)$ is also a noetherian
$\widehat{R}-$module. It follows that
$\mathfrak{F}^I_{s}(K)/I\mathfrak{F}^I_{s}(K)$ is a noetherian
$\widehat{R}-$module and the proof is complete.
\end{proof}

Let $M$ be  an artinian $R-$module. There is an other question: When
are the formal local homology modules $\mathfrak{F}^I_{i}(M)$
artinian? The following theorem gives us an answer when $i=\Ndim M.$

\begin{theorem}\label{T:ddfmnttchd}
Let  $M$ be  an artinian $R-$module with $\Ndim M = d.$ Then
$\mathfrak{F}^I_{d-1}(M)$ is an artinian $R-$module.
\end{theorem}
\begin{proof}
We prove by induction on $d = \Ndim M.$ When $d=0,$  It follows from
\ref{P:mdcpitddfmb0}
 that
$\mathfrak{F}^I_{0}(M)\cong M.$  Then $\mathfrak{F}^I_{0}(M)$ is an
artinian $R-$module.

Let $d>0.$  As $M$ is artinian, there is a positive integer $m$ such
that $I^tM = I^mM$ for all $t \geq  m.$ Set $K = I^mM,$ analysis
similar to that in the proof of \ref{L:dktthhsrdcddfm}  shows that
$$\mathfrak{F}^I_{i}(K)\cong \mathfrak{F}^I_{i}(M)$$ for all $i>0.$
Thus, The proof is completed by showing that $\mathfrak{F}^I_{d}(K)$
is an artinian $R-$module. Since $IK = K,$ there is an element $x
\in I$ such that $xK = K$ by \cite[2.8]{macsec}. Now the short exact
sequence
$$0\longrightarrow 0:_Kx \longrightarrow K \overset{x}\longrightarrow K \longrightarrow
0.$$It induces an exact sequence
$$\mathfrak{F}^I_{d-1}(0:_Kx) \longrightarrow \mathfrak{F}^I_{d-1}(K)\overset{x}\longrightarrow
\mathfrak{F}^I_{d-1}(K).$$ It should be noted by \cite[3.7]{cuoalo}
that $\Ndim(0:_K x)\leq d-1.$ Then $\mathfrak{F}^I_{d-1}(0:_Kx)$ is
an artinian $R-$module by the inductive hypothesis. From the last
exact sequence, $0:_{\mathfrak{F}^I_{d-1}(K)} x$ is an artinian
$R-$module. Moreover, $\Gamma(\mathfrak{F}^I_{d-1}(K)) =
\mathfrak{F}^I_{d-1}(K)$ by \ref{T:ddddflb0}. It follows from
\cite[1.3]{melona}  that $\mathfrak{F}^I_{d-1}(K)$ is an artinian
$R-$module and the proof is complete.
\end{proof}

The following theorem gives us the  equivalent properties for formal
local homology modules $\mathfrak{F}^I_{i}(M)$ being finitely
artinian for all $i>s.$

\begin{theorem}\label{L:lh0lhsdktthhsrdcddfm}
Let $M$ be an artinian $R-$module and $s$ a non-negative integer.
Then the following statements are equivalent:
\begin{itemize}
\item[(i)] $\mathfrak{F}^I_{i}(M)$ is artinian for all
$i>s;$
\item[(ii)] $\mathfrak{F}^I_{i}(M) = 0$ for all
$i>s;$
\item[(iii)] $\Ass(\mathfrak{F}^I_{i}(M))\subseteq \{\m\}$  for all
$i>s.$
\end{itemize}
\end{theorem}
\begin{proof}
$(i)\Rightarrow (ii).$ We use induction on $d=\Ndim M.$ If $d=0,$
then $\Ndim(0:_MI)=0.$ By \ref{C:hqdlkttdddpfm},
$\mathfrak{F}^I_{i}(M) = 0$ for all $i>0.$

Let $d>0.$ As $M$ is artinian, there is a positive integer $m$ such
that $I^tM = I^mM$ for all $t \geq  m.$ Set $K = I^mM,$ analysis
similar to that in the proof of \ref{L:dktthhsrdcddfm}  shows that
$$\mathfrak{F}^I_{i}(K)\cong \mathfrak{F}^I_{i}(M)$$ for all $i>0.$
Thus, the proof will be complete if we show that
$\mathfrak{F}^I_{i}(K) = 0$ for all $i>s.$ By the hypothesis,
$\mathfrak{F}^I_{i}(K)$ is artinian for all $i>s.$ As $IK = K,$
there is an element $x \in I$ such that $xK = K$ by
\cite[2.8]{macsec}.  Now the short exact sequence $0 \longrightarrow
0:_K x \longrightarrow K \overset{x}\longrightarrow K
\longrightarrow 0$ gives rise to a short exact sequence of formal
local homology modules
$$...\longrightarrow \mathfrak{F}^I_{i+1}(K)\longrightarrow \mathfrak{F}^I_{i}(0:_K x)\longrightarrow \mathfrak{F}^I_{i}(K)\overset{.x}\longrightarrow \mathfrak{F}^I_{i}(K)\longrightarrow....$$
By \cite[4.7]{cuoalo} $\Ndim(0:_Kx)\leq d-1.$ Then the inductive
hypothesis gives $\mathfrak{F}^I_{i}(0:_K x) = 0$ for all $i > s$
and we have an exact sequence
$$0 \longrightarrow \mathfrak{F}^I_{i}(K)\overset{.x}\longrightarrow \mathfrak{F}^I_{i}(K)$$
for all $i>s.$ It follows that $0:_{\mathfrak{F}^I_{i}(K)} x = 0$
for all $i>s.$ Since $\mathfrak{F}^I_{i}(K)$ is artinian for all
$i>s,$ we conclude that $\mathfrak{F}^I_{i}(K)=0$ for all $i>s.$

$(ii)\Rightarrow (iii)$ is trivial.

$(iii)\Rightarrow (i).$ We use induction on $d=Ndim M$ and the
argument is similar to that used in the proof of $(i)\Rightarrow
(ii).$ If $d=0,$ $\mathfrak{F}^I_{i}(M) = 0$ for all $i>0.$ Let
$d>0.$ As $M$ is artinian, there is a positive integer $m$ such that
$I^tM = I^mM$ for all $t \geq  m.$ Set $K = I^mM,$ we have
$$\mathfrak{F}^I_{i}(K)\cong \mathfrak{F}^I_{i}(M)$$ for all $i>0$
and the proof will be complete if we show that
$\mathfrak{F}^I_{i}(K)$ is artinian for all $i>s.$  Now, there is an
element $x \in I$ such that $xK = K$ and the short exact sequence $0
\longrightarrow 0:_K x \longrightarrow K \overset{x}\longrightarrow
K \longrightarrow 0$ gives rise to a long exact sequence of formal
local homology modules
$$...\longrightarrow \mathfrak{F}^I_{i+1}(K)\longrightarrow \mathfrak{F}^I_{i}(0:_K x)\longrightarrow \mathfrak{F}^I_{i}(K)\overset{.x}\longrightarrow \mathfrak{F}^I_{i}(K)\longrightarrow....$$
By \cite[4.7]{cuoalo} $\Ndim(0:_Kx)\leq d-1.$ On the other hand,
$\Ass(\mathfrak{F}^I_{i}(K))\subseteq \{\m\}$ for all $i>s$ by  the
hypothesis. From the exact sequence we have
$\Ass(\mathfrak{F}^I_{i}(0:_K x))\subseteq \{\m\}$ for all $i>s.$
Then   $\mathfrak{F}^I_{i}(0:_K x)$ is artinian for all $i>s$ by the
inductive hypothesis. Therefore $0:_{\mathfrak{F}^I_{i}(K)} x$ is
artinian for all $i>s.$ Moreover, $\Gamma_I(\mathfrak{F}^I_{i}(K)) =
\mathfrak{F}^I_{i}(K)$ by \ref{T:ddddflb0}. From \cite[1.3]{melona}
we conclude that
  $\mathfrak{F}^I_{i}(K)$ is artinian for all
$i>s$ and this finishes the proof.
\end{proof}


\begin{thebibliography}{9}\small
\bibitem{bijart} M. Bijan-Zadeh, S. Rezaei,  Artinianness and Attached Primes of Formal
Local Cohomology Modules," {\it Algebra Colloquium} 21 : 2 (2014),
307-316.
\bibitem{broloc} Brodmann M. P., Sharp R. Y., {\it Local cohomology: an algebraic introduction with geometric applications,} Cambridge University Press (1998).
\bibitem{cuothe} N. T. Cuong, T. T. Nam, \lq\lq The $I-$adic completion and local homology for Artinian
modules,\rq\rq\
 {\it   Math. Proc. Camb. Phil. Soc.,}   131 (2001), 61-72.
\bibitem{cuoalo} N. T. Cuong, T. T. Nam, A Local homology theory for linearly compact module, {\it J. Algebra} 319(2008), 4712-4737.
\bibitem{kirdim} D. Kirby,  {\it Dimension and length of artinian
modules",} Quart, J. Math. Oxford {\bf (2) 41}(1990), 419-429.
\bibitem{macdua} I. G. Macdonald, Duality over complete local
rings, {\it Topology} {\bf 1}(1962), 213-235.
\bibitem{macsec} I. G. Macdonald, Secondary representation of modules over a commutative ring, Symposia Mathematica, Vol. XI (Convegno di Algebra Commutativa, INDAM, Rome, 1971), pp. 23–43. Academic Press, London, 1973.
\bibitem{mafres} A. Mafi, "Results of formal local cohomology modules," {\it  Bull. Malays. Math. Sci. Soc.} (2) 36 (2013), no. 1,
173-177.
\bibitem{melpro} L. Melkersson, Properties of cofinite modules and applications to local cohomology, Math. Proc. Camb. Phil. Soc.(1999),125--417
\bibitem{melona} L. Melkersson, "On a symptotic stability for sets of prime ideals connected with the powers of an ideal,"
{\it   Math. Proc. Camb. Phil. Soc.}  107(1990), 267-271.
\bibitem{nammin} T. T. Nam, \lq\lq Minimax modules, local homology and local cohomology,\rq\rq\   {\it
International Journal of Mathematics} Vol. 26, No. 12 (2015) 1550102
(16 pages).
\bibitem{robkru} R. N. Roberts, {\it Krull dimension for artinian
modules over quasi-local commutative rings,} Quart. J. Math. Oxford
(3) {\bf 26}(1975), 269-273.

\bibitem{shaame} R. Y. Sharp.  {\it A method for the study of artinian
modules with an application to asymptotic behavior}.  In:
Commutative Algebra (Math. Sciences Research Inst. Publ. No. 15,
Springer-Verlag, 1989), 443-465.

\bibitem{schonf} P. Schenzel, On formal local cohomology and connectedness, \textit{J. Algebra}, 315(2), (2007), 894-923

\bibitem{strhom} Strooker J., {\it Homological questions in local algebra,} Cambridge University Press
(1990).

\bibitem{yanart} Yan Gu, "The Artinianness Of Formal Local Cohomology
Modules," {\it Bull. Malays. Math. Sci. Soc.} (2) 37 (2014), no. 2,
449-456.
\bibitem{yascoa} S. Yassemi, Coassociated primes, Comm. Algebra 23 (4) (1995) 1473-1498.
\bibitem{zoslin} H. Z\"oschinger, Linear-Kompakte Moduln uber noetherschen Ringen, {\it Arch. Math.}  {\bf 41}(1983), 121-130.
\end{thebibliography}
\end{document}